\documentclass[12pt, leqno, british]{amsart}
\usepackage[
	style=alphabetic,
	citestyle=alphabetic,
	backend=biber
]{biblatex}

\usepackage{harmony}
\usepackage{a4, amsmath}
\usepackage{mathtools}
\usepackage{amssymb}
\usepackage{epsfig}
\usepackage{graphics}
\usepackage{amsthm, amscd, mathdots}
\swapnumbers
\usepackage{enumerate}
\usepackage{url}
\usepackage{csquotes}
\usepackage{color}
\usepackage{datetime}
\usepackage{stmaryrd}
\usepackage{ifthen}

\usepackage[pass]{geometry}

\usepackage[all]{xy}

\usepackage{hyperref}
\usepackage{cleveref}

\theoremstyle{definition}
\newtheorem{defi}{Definition}[section]

\theoremstyle{plain}
\newtheorem{prop}[defi]{Proposition}
\newtheorem{lem}[defi]{Lemma}
\newtheorem{stel}[defi]{Theorem}
\newtheorem{gev}[defi]{Corollary}

\newtheorem{ques}[defi]{Question}
\newtheorem*{stel*}{Theorem}
\theoremstyle{remark}
\newtheorem{opm}[defi]{Remark}
\newtheorem{vb}[defi]{Example}
\newtheorem{vbn}[defi]{Examples}

\newcommand{\nat}{\mathbb{N}}
\newcommand{\zz}{\mathbb{Z}}
\newcommand{\rr}{\mathbb{R}}
\newcommand{\ff}{\mathbb{F}}

\newcommand{\qq}{\mathbb{Q}}
\newcommand{\charac}{\mathsf{char}}
\newcommand{\mbb}{\mathbb}

\newtheorem*{prop*}{Proposition}

\newcommand{\llangle}{\langle\!\langle}
\newcommand{\rrangle}{\rangle\!\rangle}
\newcommand{\mc}{\mathcal}
\newcommand{\mf}{\mathfrak}

\newcommand{\ovl}{\overline}
\newcommand{\RF}[2]{\mathrm{r}_{#1}(#2)}
\DeclareMathOperator{\Br}{Br}

\let\dim\relax
\DeclareMathOperator{\dim}{\mathsf{dim}}

\renewcommand{\min}{\mathsf{min}}
\renewcommand{\bmod}{\,\,\mathsf{mod}\,\,}
\renewcommand{\setminus}{\smallsetminus}
\renewcommand{\leq}{\leqslant}
\renewcommand{\geq}{\geqslant}

\DeclareMathOperator{\Span}{span}
\DeclareMathOperator{\trdeg}{trdeg}
\DeclareMathOperator{\rrk}{rrk}

\newcommand{\IS}[1]{I_q^{[#1]}}
\newcommand{\Li}[1]{\mbb{L}^{#1}}

\newcommand{\pow}[1]{^{(#1)}}

\makeatletter
\newcommand{\bigperp}{%
  \mathop{\mathpalette\bigp@rp\relax}%
  \displaylimits
}

\newcommand{\bigp@rp}[2]{%
  \vcenter{
    \m@th\hbox{\scalebox{\ifx#1\displaystyle2.1\else1.5\fi}{$#1\perp$}}
  }%
}
\makeatother

\title{Linkage of Pfister forms over semi-global fields}
\author{Nicolas Daans}
\date{20.08.2024}
\address{Charles University, Faculty of Mathematics and Physics, Department of Algebra, Sokolov\-sk\' a 83, 18600 Praha~8, Czech Republic}
\email{nicolas.daans@matfyz.cuni.cz}
\thanks{
The author gratefully acknowledges support by {Czech Science Foundation} (GA\v CR) grant 21-00420M, and {Charles University} PRIMUS Research Programme PRIMUS/24/SCI/010. \\
This version of the article has been accepted for publication after peer review, but is not the Version of Record and does not reflect post-acceptance improvements, or any corrections. The Version of Record is available online at: \href{https://doi.org/10.1007/s00209-024-03598-2}{https://doi.org/10.1007/s00209-024-03598-2}.
  }

\begin{document}
\begin{abstract}
We study linkage of $(d+1)$-fold quadratic Pfister forms over function fields in one variable over a henselian valued field of 2-cohomological dimension $d$.
Specifically, we characterise this property in terms of linkage of quadratic Pfister forms over function fields over the residue field of the henselian valued field; in full generality in characteristic different from 2, and for most complete discretely valued fields in characteristic 2.
As an application, we obtain a proof that $(d+2)$-fold quadratic Pfister forms over function fields in one variable over a $d$-dimensional higher local field are linked.


\medskip\noindent
{\sc Classification (MSC 2020):} 11E81 (primary), 11E04, 11R58 (secondary)
\end{abstract}
\maketitle

\section{Introduction}
Over a global field $K$ (like $\qq$, the field of rational numbers) any pair of quaternion algebras is linked, i.e.~the two algebras have a common quadratic subfield.
This can be inferred, for example, from the Hasse-Minkowski Theorem.
This fact implies that the sum of any finite number of classes of quaternion algebras in the Brauer group $\Br(K)$ is again represented by a quaternion algebra. 
Furthermore, this fact can be expressed very explicitly in the language of quadratic forms as follows, assuming for simplicity that $\charac(K) \neq 2$: for all non-zero elements $a_1, a_2, b_1, b_2 \in K$ there exist non-zero elements $c_1, c_2, d \in K$ such that
\begin{align*}
a_1X^2 + a_2Y^2 - a_1a_2Z^2 &\cong c_1X^2 + dY^2 - c_1dZ^2 \quad \text{and} \\
b_1X^2 + b_2Y^2 - b_1b_2Z^2 &\cong c_2X^2 + dY^2 - c_2dZ^2,
\end{align*}
where the relation $\cong$ denotes isometry of quadratic forms, i.e. the forms can be obtained from one another via a linear change of variables.

This statement does not hold when $K$ is replaced by an arbitrary field.
For example, if $K$ is a finitely generated extension of transcendence degree $1$ of a non-real global field (like $\qq(\sqrt{-1})$), then there exist non-linked pairs of quaternion algebras.
This can be partially explained through the fact that $K$ has cohomological dimension $3$, whereas non-real global fields have cohomological dimension $2$.
Over $K$, we should thus rather study the correct higher-dimensional analogues to quaternion algebras.
These are formed by so-called $3$-fold Pfister forms, which are quadratic forms in $2^3 = 8$ variables determined by $3$ parameters from the base field.
And indeed, $3$-fold Pfister forms over $K$ are linked, in the sense that for any pair of them, one can find representations which share $2$ of their $3$ defining parameters.
This fact is contained essentially in recent work of Suresh \cite{Suresh_ThirdGalCohom}; see \Cref{E:linkage} and the discussion afterwards.

In general, for a field $K$ and natural number $d$, $d$-fold Pfister forms (sometimes called $d$-fold \textit{quadratic} Pfister forms) over $K$ form a natural class of quadratic forms in $2^d$ variables over $K$. We say that $d$-fold Pfister forms over $K$ are \emph{linked} if any pair of $d$-fold Pfister forms over $K$ contains a common $(d-1)$-fold Pfister form as a subform.
We refer to sections \ref{sect:QF-preliminaries} and \ref{sect:linkage} below for the precise definitions of these and other concepts from quadratic form theory over fields, as well as for known examples and references.

In this note, we study linkage of $d$-fold Pfister forms over fields, in particular function fields, where $d$ is the $2$-cohomological dimension of the field, or the appropriate analogue of this notion in characteristic $2$.
%
More specifically, we will look at function fields in one variable (i.e.~finitely generated extensions of transcendence degree $1$) over henselian valued fields.
When this valued field is additionally complete and discrete, these function fields are sometimes called \emph{semi-global fields}, hence the title of this text.

We will give a new proof of the known fact that $3$-fold Pfister forms over a function field in one variable over a $p$-adic field are linked, without relying on the cohomological techniques from \cite{Parimala-Suresh-u-invariant,Suresh_ThirdGalCohom} or the highly technical zero counting arguments used in \cite{HeathBrownZeroes, Leep_uInvariantPAdicFuncField}.
Our approach instead relies on an abstract study of how linkage properties for quadratic forms over function fields can be lifted from residue fields of henselian valued fields to the valued fields themselves.
By iterating this, we in fact obtain a generalisation to function fields over so-called ``higher local fields'':
\begin{stel*}[see \Cref{T:higher-local}]
Let $d$ be a non-zero natural number and suppose that there is a sequence $K_1, K_2, \ldots, K_d$ of fields, where $K_1$ is finite, and each $K_{i+1}$ is complete with respect to a discrete valuation with residue field $K_i$.
For any function field in one variable $F/K_i$, $(d+1)$-fold Pfister forms over $F$ are linked.
\end{stel*}
It should be noted that, when $\charac(K_1) \neq 2$ in the above Theorem, the result can also be derived from existing work on the strong $u$-invariant of semi-global fields developed in \cite{HHK_ApplicationsPatchingToQuadrFormsAndCSAs}; see \Cref{R:local-global-technique} below for a discussion.
The case where for example $\charac(K_2) = 2 \neq \charac(K_d)$ in the above Theorem does not seem to follow in any straightforward way from previously developed techniques.

Let us phrase the more general transfer results for function fields over henselian valued fields which we obtain.
For a natural number $d$, let us call a field $K$ an \emph{$\Li{d}$-field} if, for any function field in one variable $F/K$, $(d+2)$-fold Pfister forms over $F$ are isotropic, and $(d+1)$-fold Pfister forms over $F$ are linked.
For a valued field $(K, v)$, let us denote by $vK$ and $Kv$ the value group and residue field, respectively; we write the operation on the value group additively.
We call a valued field $(K, v)$ \textit{dyadic} if $\charac(Kv) = 2$ (equivalently, if $v(2) > 0$), and \textit{non-dyadic} otherwise.
We obtain the following principle for function fields over arbitrary non-dyadic henselian valued fields.
\begin{stel*}[see \Cref{T:classKinduction}]
Let $d$ be a natural number, let $(K, v)$ be a henselian valued field with $\charac(Kv) \neq 2$.
Assume that $m = [vK : 2vK] < \infty$.
If $Kv$ is an $\Li{d}$-field, then $K$ is an $\Li{d+m}$-field.
\end{stel*}
As mentioned earlier, $3$-fold Pfister forms over function fields in one variable over $\qq(\sqrt{-1})$ are linked, making $\qq(\sqrt{-1})$ an $\Li{2}$-field.
The above Theorem implies, for example, that $\qq(\sqrt{-1})(\!(T)\!)$, the field of formal Laurent series over $\qq(\sqrt{-1})$, is an $\Li{3}$-field (see \Cref{E:Ld-fields-new}).
In particular, $4$-fold Pfister forms over the rational function field $\qq(\sqrt{-1})(\!(T)\!)(X)$ are linked.

For function fields in characteristic $0$ over dyadic henselian discretely valued fields, we obtain a similar result under mild assumptions on the degree of imperfection of the residue field (see \Cref{P:AllSymbolsChar2} below).
For a field $K$, denote by $K\pow{2}$ its subset of squares; this is a subfield if $\charac(K) = 2$.
\begin{stel*}[see \Cref{T:classKinduction-char2}]
Let $(K, v)$ be a henselian discretely valued field with $\charac(K) = 0$ and $\charac(Kv) = 2$.
Assume that $d = \log_2([Kv : Kv\pow{2}])-1 < \infty$.
Then $K$ is an $\Li{d+1}$-field.
\end{stel*}
If $\charac(K) = 2$ then no such general theorem can exist and only specific cases can be covered, see \Cref{R:complete-char2} and \Cref{P:complete-char2} below.

Our proofs rely on the usage of a recent local-global principle for isotropy of quadratic forms developed in \cite{Mehmeti_PatchingBerkQuad}, which reduces the problem to the study of linkage over a certain family of henselian valued fields.
In the non-dyadic case, linkage over henselian valued fields can be understood completely via the theory of residue forms, and in the dyadic case, residue forms still provide important pieces of the puzzle.
The used residue forms techniques are largely standard, but for lack of a good reference for valued fields with arbitrary value groups and residue characteristics, we discuss them in detail in \Cref{sect:residue-forms} and then apply them to study linkage over henselian valued fields in \Cref{sect:linkage-local}.

In the dyadic case, our understanding of quadratic forms over henselian valued fields is very limited, and our proofs rely on a delicate combination of partial results on the classification of quadratic forms over dyadic henselian valued fields from \cite{Tietze,MMW91,SpringerTameQuadratic,ChapmanMcKinnieSymbolLength,Andromeda-1}.

\subsection*{Acknowledgements}
I want to thank Karim Johannes Becher for valuable input, in particular suggesting a version of \Cref{P:d-linked-characterisation} as well as pointing out \Cref{vb-Li_n-vs-cd2}.

The questions which inspired me to write this paper came up in the margin of a joint project with Philip Dittmann \cite{Andromeda-2} and I want to thank him for several valuable comments on earlier drafts of this manuscript.

Finally, I am grateful to the anonymous referee, whose diligent proofreading helped identify and resolve several technical issues in an earlier draft of this paper.

\section{Preliminaries on quadratic forms}\label{sect:QF-preliminaries}
Let always $K$ be a field.
We recall and establish some terminology and notation, and refer to \cite[Chapter I-II]{ElmanKarpenkoMerkurjev} for details.
We will denote by $\nat$ the set of natural numbers, and by $\nat^+$ the proper subset of non-zero natural numbers.

A \emph{bilinear form over $K$} is a $K$-bilinear map $\mf{b} : V \times V \to K$, where $V$ is a finite-dimensional $K$-vector space.
We call a bilinear form $\mf{b}$ \emph{non-singular} if for all $x \in V \setminus \lbrace 0 \rbrace$ there exists $y \in V \setminus \lbrace 0 \rbrace$ with $\mf{b}(x, y) \neq 0$.

A \emph{quadratic form over $K$} is a map $q : V \to K$ where $V$ is a finite-dimensional $K$-vector space, such that $q(av) = a^2q(v)$ for all $v \in V$ and $a \in K$, and such that the map
$$ \mf{b}_q : V \times V \to K : (v, w) \mapsto q(v+w) - q(v) - q(w)$$
is a bilinear form.
A quadratic form $q$ is called \emph{non-singular} if its associated bilinear form $\mf{b}_q$ is non-singular.
Given a quadratic form $q : V \to K$, we call $\dim(V)$ the \emph{dimension of $q$}, and might also denote this quantity by $\dim(q)$.

We remark that, if $q : V \to K$ is a quadratic form and $(v_1, \ldots, v_n)$ is a basis for $V$, then there exists a homogeneous degree $2$ polynomial $f(X_1, \ldots, X_n) \in K[X_1, \ldots, X_n]$ such that $q(x_1v_1 + \ldots + x_nv_n) = f(x_1, \ldots, x_n)$ for all $(x_1, \ldots, x_n) \in K^n$.
Conversely, given an arbitrary homogeneous degree $2$ polynomial $f(X_1, \ldots, X_n) \in K[X_1, \ldots, X_n]$, the map $(x_1, \ldots, x_n) \mapsto f(x_1, \ldots, x_n)$ defines a quadratic form $K^n \to K$.
Our definition of quadratic forms as such gives a coordinate-free version of the perhaps more usual definition ``a quadratic form is a homogeneous polynomial of degree 2'' which was alluded to in the Introduction.
While for examples it can be convenient to use the explicit polynomial framework, general proofs can often be more conceptually lean in the coordinate-free framework.
The polynomial framework will however be useful in Propositions \ref{P:binaryFormValuations} to  \ref{P:Pfister-residue-computation}.

Given a quadratic form $q : V \to K$, we call $q$ \emph{isotropic} if $q(v) = 0$ for some $v \in V \setminus \lbrace 0 \rbrace$, \emph{anisotropic} otherwise.
We call a subspace $W \subseteq V$ \emph{totally isotropic} (for $q$) if $q(v) = 0$ for all $v \in W$.
If $q$ is non-singular, we denote by $i_W(q)$ the maximal dimension of a totally isotropic subspace of $q$, called the \emph{Witt index of $q$}.
The form $q$ is called \emph{hyperbolic} if it is non-singular and $\dim(q) = 2i_W(q)$.

Given quadratic forms $q$ and $q'$ over $K$, we denote $q \cong q'$ to say that they are \emph{isometric}.
We denote the (external) \emph{orthogonal sum} of $q$ and $q'$ by $q \perp q'$.
Similarly, we use the symbols $\cong$ and $\perp$ also for isometry and orthogonal sum of bilinear forms.
See \cite[Sections 1 and 7]{ElmanKarpenkoMerkurjev} for details.
We say that two quadratic forms $q$ and $q'$ are \emph{similar} if $q \cong aq'$ for some $a \in K^\times$.
Finally, for a quadratic form $q$ and a bilinear form $\mf{b}$, we denote by $\mf{b} \otimes q$ the \emph{Kronecker product} (or \emph{tensor product}), as defined in \cite[Section 8]{ElmanKarpenkoMerkurjev}.

Given $d \in \nat$ and $a_1, \ldots, a_d \in K^\times$, the \emph{$d$-fold bilinear Pfister form} $\llangle a_1, \ldots, a_d \rrangle^b_K$ is a $K$-bilinear form on $K^{2^d}$.
Such $d$-fold bilinear Pfister forms can be defined inductively: the $0$-fold bilinear Pfister form $\llangle\rrangle^b_K$ is the bilinear form on the $1$-dimensional vector space $K$ given by
$$ K \times K \to K : (x, y) \mapsto xy. $$
Given a $d$-fold bilinear Pfister form $B = \llangle a_1, \ldots, a_d \rrangle^b_K : V \times V \to K$ and $a_{d+1} \in K^\times$, $\llangle a_1, \ldots, a_{d+1} \rrangle^b_K$ is defined by the mapping
$$ (V \times V) \times (V \times V) \to K : ((v_1 , v_2), (w_1 , w_2)) \mapsto B(v_1, w_1) - a_{n+1}B(v_2, w_2). $$
We refer to \autocite[Section 6]{ElmanKarpenkoMerkurjev} for more details.

Now we fix $a_1, \ldots, a_d \in K^\times$ and $b \in K$ with $1+4b \neq 0$ and define the \emph{$(d+1)$-fold quadratic Pfister form} $\llangle a_1, \ldots, a_d, b ]]_K$, which is a quadratic form on $K^{2^{d+1}}$.
We first define the $1$-fold Pfister form $\llangle b ]]_K$ as the quadratic form
$$ K^2 \to K : (x, y) \mapsto x^2 - xy - by^2. $$
Now we define $\llangle a_1, \ldots, a_d, b]]_K$ as the tensor product $\llangle a_1, \ldots, a_d \rrangle^b_K \otimes \llangle b ]]_K$.

In this paper, we will drop the adjective `quadratic' and simply refer to quadratic Pfister forms as \emph{Pfister forms}; we will add the adjective `bilinear' when specifically referring to bilinear Pfister forms.
See \autocite[Section 9.B]{ElmanKarpenkoMerkurjev} for more on quadratic Pfister forms; we remark that the parametrisation of Pfister forms used in \cite{ElmanKarpenkoMerkurjev} for fields of characteristic not $2$ is different from ours, but leads to the same class of quadratic forms up to isometry (see also \cite[Remark 3.2]{Andromeda-1}).

We will consider the \emph{(quadratic) Witt group}\index{Witt group, Witt ring, Witt equivalent} $I_qK$ of a field $K$; see \autocite[Section 8]{ElmanKarpenkoMerkurjev} for proofs and details.
The isometry classes of even-dimensional non-singular quadratic forms form a commutative monoid under the orthogonal sum. We define $I_qK$ to be the monoid obtained by taking the quotient modulo the submonoid of isometry classes of hyperbolic forms. Since $q \perp -q$ is hyperbolic for any non-singular form $q$ over $K$, $I_qK$ is an abelian group.
Non-singular quadratic forms which represent the same element of $I_qK$ are called \emph{Witt equivalent}. Witt equivalent quadratic forms are isometric if and only if they have the same dimension.
Every Witt equivalence class contains a unique anisotropic quadratic form.
For a non-singular quadratic form $q$ over $K$, we denote by $[q]$ its equivalence class in $I_qK$.

Pfister forms are non-singular.
For any $d \geq 2$, we denote by $I_q^dK$ the subgroup of $I_qK$ generated by the classes of $d$-fold Pfister forms, and we set $I_q^1K = I_qK$.
We have $ I_q^1K \supseteq I_q^2K \supseteq I_q^3K \supseteq \ldots$, and furthermore that $I_q^dK$ contains the class of every quadratic form similar to a $d$-fold Pfister form \cite[Section 9.B]{ElmanKarpenkoMerkurjev}; we will call a form which is similar to a $d$-fold Pfister form a \emph{scaled $d$-fold Pfister form}.
We will sometimes coloquially speak of a ``quadratic form in $I_q^dK$'' when we mean a quadratic form whose equivalence class lies in $I_q^dK$.

If $K$ is a field with $\charac(K) \neq 2$, then $I_qK$ can be viewed as an ideal of the so-called \emph{Witt ring} $WK$ of $K$.
We are only interested in the additive group structure of $WK$, so we will not discuss the multiplication operation.
We refer to \autocite[Section 2]{ElmanKarpenkoMerkurjev} for the precise construction of $WK$, and to \autocite[Section 7]{ElmanKarpenkoMerkurjev} for its relation to quadratic forms.
$WK$ may be thought of as constructed by taking the monoid of isometry classes of non-singular quadratic forms (under the orthogonal sum) and taking the quotient modulo the submonoid of isometry classes of hyperbolic forms.
Also in this situation, we call non-singular quadratic forms \emph{Witt equivalent} if they correspond to the same element of $WK$. Witt equivalent forms are isometric if and only if they have the same dimension.
Every Witt equivalence class contains a unique anisotropic quadratic form.

Finally, when $L/K$ is a field extension, we can associate to any quadratic form $q : V \to K$ over $K$ a quadratic form $q_L : V \otimes_K L \to L$ such that $q_L(x \otimes a) = a^2q(x)$ and $\mf{b}_{q_L}(x \otimes a, y \otimes b) = ab\mf{b}_q(x, y)$ for all $x,y \in V, a,b \in L$, the \emph{scalar extension} of $q$ to $L$.
This induces a group homomorphism $I_q K \to I_q L$ called the \emph{restriction map}, and it maps the class of a $d$-fold Pfister form to the class of a $d$-fold Pfister form.

\section{Linkage of Pfister forms}\label{sect:linkage}
\begin{defi}\label{D:Linkage}
For $d \in \nat^+$, we say that \emph{$I^d_qK$ is linked} if every element of $I^d_qK/I^{d+1}_qK$ is the class of a $d$-fold Pfister form.
\end{defi}

The following proposition gives an alternative characterisation of linkage which is often given as a definition.
See \autocite[Section 24]{ElmanKarpenkoMerkurjev} for more on linkage of Pfister forms.
\begin{prop}\label{P:LinkageCharacterisation}
Let $K$ be a field, $d \in \nat^+$.
We have that $I^d_q K$ is linked if and only if for any two $d$-fold Pfister forms $q_1, q_2$ over $K$ there exist $a_1, \ldots, a_{d-1}, b, a_1' \in K$ with $a_1 \cdots a_{d-1}(1+4b)a_1' \neq 0$ such that $$q_1 \cong \llangle a_1, a_2 \ldots, a_{d-1}, b]]_K \enspace\text{and}\enspace q_2 \cong \llangle a_1', a_2, \ldots, a_{d-1}, b]]_K.$$
\end{prop}
\begin{proof}
If $q_1 \cong \llangle a_1, a_2 \ldots, a_{d-1}, b]]_K$ and $q_2 \cong \llangle a_1', a_2, \ldots, a_{d-1}, b]]_K$ for $a_1, \ldots, a_{d-1}, b, a_1' \in K$ with $a_1 \cdots a_{d-1}(1+4b)a_1' \neq 0$, then $q_1 \perp q_2$ is equivalent modulo $I^{n+1}_qK$ to the quadratic form $\llangle a_1a_1', a_2, \ldots, a_{d-1}, b]]_K$ \autocite[Example 4.10]{ElmanKarpenkoMerkurjev}.
Thus, if such representations can be found for any two $d$-fold Pfister forms, and since $I^d_qK/I^{d+1}_qK$ is generated by classes of $d$-fold Pfister forms, we infer that $I^d_q K$ is linked.

The other implication follows from \autocite[Proposition 24.5]{ElmanKarpenkoMerkurjev}.
\end{proof}
In other words, \Cref{P:LinkageCharacterisation} says that $I^d_qK$ is linked if and only if any two $d$-fold quadratic Pfister forms over $K$ contain a common $(d-1)$-fold Pfister form as a subform.
This is what we described as ``$d$-fold Pfister forms over $K$ are linked'' in the Introduction.

The notion of linkage of Pfister forms in characteristic not $2$ was introduced in \cite[Section 4]{ElmanLamLinkage}.
Meanwhile, several variations of this notion have been considered.
For example, one can study when a pair of quadratic $d$-fold Pfister forms contains a common $(d-m)$-fold quadratic Pfister form as a subform for different values of $m$, see e.g.~\cite[Section 24]{ElmanKarpenkoMerkurjev}.
In characteristic $2$, one can also consider the property of a pair of $d$-fold Pfister forms to have a common bilinear $(d-1)$-fold Pfister form as a factor; by \cite[Corollary 2.1.4]{FaivreThesis} this is stronger than our notion introduced in \Cref{D:Linkage}.
See \cite{CGV-Linkage,CD-Linkage} for recent work on notions of linkage in characteristic $2$.
Finally, we mention that the notion of ``$n$-linkage of $d$-fold Pfister forms'' (i.e.~for any set containing $n$ $d$-fold Pfister forms over $K$ there exists a common $(d-1)$-fold Pfister form as a subform) has gained significant attention in recent years, see e.g.~\cite{Becher-Triple,CD-Triple,CT-LinkagePfister,Chapman-Linkage-quaternion,BGStrongLinkage}.

\begin{defi}
Let $d \in \nat^+$. 
We call a field $K$ \emph{top-$d$-linked} if $I^d_q K$ is linked and $I^{d+1}_q K = 0$.
We say that a field $K$ is an \emph{$\Li{d}$-field} if every function field in one variable $F/K$ is top-$(d+1)$-linked.
\end{defi}
Here, by a function field in one variable, we mean a finitely generated field extension of transcendence degree one.
Since the property of being top-$(d+1)$-linked is clearly of a finitary nature, one has that a field $K$ is an $\Li{d}$-field if and only if every transcendence degree one field extension $F/K$ (not necessarily finitely generated) is top-$(d+1)$-linked.
However, the definition with function fields will be easier to work with.
\begin{prop}\label{P:cdFacts}
Let $K$ be a field with $\charac(K) \neq 2$, $d \in \nat$.
If $I^{d+1}_q L = 0$ for all finite separable extensions $L/K$, then $I^{d+2}_{q} F = 0$ for all field extensions $F/K$ of transcendence degree $1$.
\end{prop}
\begin{proof}
This can be translated to a well-known property of the $2$-cohomological dimension of fields in view of the Milnor conjectures.
See for example \cite[Proposition 3.5]{BDGMZ}: there the hypothesis ``$I^{d+2}_q K(X) = 0$'' is used instead of ``$I^{d+1}_q L = 0$ for all finite separable extensions $L/K$'', but in \cite[Remark 3.4]{BDGMZ} it is explained that these are equivalent.
\end{proof}
\begin{prop}\label{P:AllSymbolsChar2}
Let $K$ be a field with $\charac(K) = 2$, $d \in \nat$.
If $[K: K\pow{2}] \leq 2^d$, then $K$ is an $\Li{d+1}$-field.
Conversely, if $K$ is an $\Li{d+1}$-field, then $[K:K\pow{2}] \leq 2^{d+1}$.
\end{prop}
\begin{proof}
If $F$ is a function field in one variable over $K$, then $[F: F\pow{2}] = 2[K : K\pow{2}]$.
Hence, for the first statement, it suffices to show that $[K: K\pow{2}] \leq 2^d$ implies that $K$ is top-$(d+1)$-linked.
This is immediate from \autocite[Remark 3.1]{ChapmanMcKinnieSymbolLength}.

For the second statement, assume $[K : K\pow{2}] > 2^{d+1}$.
Let $a_1, \ldots, a_{d+2} \in K^\times$ be linearly independent over $K\pow{2}$.
The $(d+3)$-fold Pfister form $\llangle a_1, \ldots, a_{d+2}, T ]]_{K(T)}$ is anisotropic, whence $I_q^{d+3} K(T) \neq 0$, whereby $K(T)$ is not top-$(d+2)$-linked and hence $K$ is not an $\Li{d+1}$-field.
\end{proof}
\begin{prop}\label{P:d-linked-characterisation}
Let $d \in \nat^+$, let $K$ be a field.
The following are equivalent.
\begin{enumerate}[(i)]
\item\label{it:d-linked-1} $K$ is top-$d$-linked.
\item\label{it:d-linked-2} Every anisotropic quadratic form in $I^d_q K$ has dimension exactly $2^d$.
\item\label{it:d-linked-3} Every anisotropic quadratic form in $I^d_q K$ has dimension less than $2^{d+1}-2^{d-1}$.
\end{enumerate}
\end{prop}
\begin{proof}
\eqref{it:d-linked-1} $\Rightarrow$ \eqref{it:d-linked-2} follows from the definition of top-$d$-linked.

\eqref{it:d-linked-2} $\Rightarrow$ \eqref{it:d-linked-3} is obvious.

\eqref{it:d-linked-3} $\Rightarrow$ \eqref{it:d-linked-1}: If there would be an anisotropic $(d+1)$-fold Pfister form over $K$, then this would be a $2^{d+1}$-dimensional anisotropic quadratic form in $I^{d}_qK$, contradicting the assumption.

Let $q_1$ and $q_2$ be $d$-fold Pfister forms over $K$. Then the class of $q_1 \perp -q_2$ lies in $I^d_qK$, whence $i_W(q_1 \perp -q_2) > 2^{-1}(\dim(q_1 \perp -q_2) - (2^{d+1} - 2^{d-1})) = 2^{d-2}$. By \autocite[Theorem 24.2 and Corollary 24.3]{ElmanKarpenkoMerkurjev} $q_1$ and $q_2$ share a subform which is a $(d-1)$-fold Pfister form. In view of \autocite[Proposition 24.1]{ElmanKarpenkoMerkurjev} and \Cref{P:LinkageCharacterisation} we obtain that $I^d_q K$ is linked.
\end{proof}

We conclude the section with some known classes of examples of $\Li{d}$-fields of characteristic not $2$ for some $d$.
Some of the references we give do not explicitly make claims about linkage of Pfister forms, but instead talk about linkage of symbols in Milnor K-theory, or of symbols in Galois cohomology.
One can however straightforwardly translate linkage statements between these different setups in view of the resolved Milnor conjectures, which state that, for a field $K$ with $\charac(K) \neq 2$ and $n \in \nat^+$, there are canonical group isomorphims between the Milnor $K$-group (modulo $2$) $k_n K$, the Galois cohomology group $H^n(K, \zz/2\zz)$, and the group $I^n_q K/I^{n+1}_q K$, and these correspondences map symbols to symbols.
See \cite[Sections 5 and 16]{ElmanKarpenkoMerkurjev} for more on the history and the precise statements of the correspondences.
\begin{vbn}\label{E:linkage}
Let $K$ be a field with $\charac(K) \neq 2$, $d \in \nat$.
\begin{enumerate}
\item\label{it:global} If $K$ is a non-real global field, then $K$ is an $\Li{2}$-field.
\item\label{it:local} If $K$ is a $p$-adic field, then $K$ is an $\Li{2}$-field.
\item\label{it:PAC} If $K$ is pseudo-algebraically closed, then $K$ is an $\Li{1}$-field.
\item\label{it:Cd} If $K$ is a $C_d$-field (see e.g.~\cite[Section 21.2]{Fri08}), then $K$ is an $\Li{d}$-field.
In particular, if $K$ is algebraically closed, then $K$ is an $\Li{0}$-field, and if $K$ is finite, then $K$ is an $\Li{1}$-field.
\end{enumerate}
\end{vbn}
\begin{proof}
\eqref{it:global}
Let $F/K$ be a function field in one variable.
Since $I^3_q L = 0$ for all non-real global fields $L$, it follows by \Cref{P:cdFacts} that $I^4_q F = 0$.
\autocite[Theorem 1.1]{Suresh_ThirdGalCohom} states that every element of $H^3(F, \zz/2\zz)$ is a symbol, which in view of the Milnor conjectures translates to the desired statement.

\eqref{it:local} If $K$ is a $p$-adic field and $F/K$ a function field in one variable, then every $9$-dimensional quadratic form over $F$ is isotropic.
This was first shown for $p \neq 2$ in \cite{Parimala-Suresh-u-invariant}, although the proof contained a gap which was rectified in the appendix to \cite{Parimala-Suresh-Degree3}.
The case for $p \neq 2$ was also solved in \cite[Corollary 4.15]{HHK_ApplicationsPatchingToQuadrFormsAndCSAs}.
In \autocite[Theorem 3.4]{Leep_uInvariantPAdicFuncField} a proof is given in the general case, and furthermore in \cite[Theorem 4]{Parimala-Suresh-PeriodIndex} a proof is given specifically for the case $p=2$.
Regardless of which proof one prefers, the linkage of $I^3_q F$ then follows from \Cref{P:d-linked-characterisation}.
Alternatively, one can argue just as in \eqref{it:global}, since \autocite[Theorem 1.1]{Suresh_ThirdGalCohom} also shows in this case that every element of $H^3(F, \zz/2\zz)$ is a symbol.
In \Cref{T:higher-local} later on, we shall give yet another proof for the linkage of $I_q^3 F$, not using either of the above two arguments.

\eqref{it:PAC}
See \autocite[Theorem 5.4]{BGStrongLinkage} for a proof of linkage in the language of Milnor $K$-theory; this translates to the desired statement in view of the Milnor conjectures.

\eqref{it:Cd}
If $K$ is a $C_d$-field, then every function field in one variable $F/K$ is a $C_{d+1}$-field, whereby every $(2^{d+1}+1)$-dimensional quadratic form over such a field $F$ is isotropic, and the linkage statement follows from \Cref{P:d-linked-characterisation}.
\end{proof}

\section{Valued fields and residue forms}\label{sect:residue-forms}
In this section we will consider fields carrying a valuation.
When $K$ is a field and $v$ a valuation on $K$, we call the pair $(K, v)$ a valued field.
We denote by $vK$ the value group of $v$, by $\mc{O}_v$ the valuation ring, by $\mf{m}_v$ the maximal ideal of $\mc{O}_v$, and by $Kv$ the residue field $\mc{O}_v/\mf{m}_v$ of $v$.
We set $2vK = \lbrace 2\gamma \mid \gamma \in vK \rbrace$.
For an element $x \in \mc{O}_v$, we denote by $\ovl{x}^v$ the element $x + \mf{m}_v$ in $Kv$; we might simply write $\ovl{x}$ if there is no risk of confusion.
Similarly, when $f \in \mc{O}_v[X_1, \ldots, X_n]$ is a polynomial, $\ovl{f}^v$ (or simply $\ovl{f}$) denotes the polynomial in $Kv[X_1, \ldots, X_n]$ obtained by reducing the coefficients modulo $\mf{m}_v$.
When $\charac(Kv) = 2$, we call $(K, v)$ (or just $v$) \emph{dyadic}, otherwise we call it \emph{non-dyadic}.
We call a valued field \emph{henselian} if the valuation extends uniquely to every finite field extension; see \cite[Section 4.1]{Eng05} for equivalent characterisations.
A finite extension $(L,w)/(K,v)$ of valued fields is called \emph{inert} (some authors use the term \emph{unramified}) if the induced residue field extension $Lw/Kv$ is separable and $[Lw : Kv] = [L : K]$, and an algebraic extension $(L, w)/(K, v)$ will be called inert if every finite subextension is inert.
We write $K_v$ for the henselisation of $(K, v)$ (see \cite[Section 5.2]{Eng05}) and denote also by $v$ a fixed extension of $v$ to $K_v$.
For a henselian valued field $(K, v)$, we will occasionally talk about the maximal inert extension $K^{un}$ of $K$ (called the \emph{inertia field} in \cite[Section 5.2]{Eng05} but otherwise often called the \emph{maximal unramified extension}); it will always be clear from the context with respect to which valuation this is done.
We call a valuation \emph{discrete} if its value group is isomorphic to $\zz$.

Let $(K, v)$ be a henselian valued field. 
Let $q : V \to K$ be an anisotropic quadratic form defined over $K$.
For $\lambda \in vK$, define the following subsets of $V$:
\begin{displaymath}
V_\lambda = \lbrace x \in V \mid v(q(x)) \geq \lambda \rbrace \qquad \text{and} \qquad V^\circ_\lambda = \lbrace x \in V \mid v(q(x)) > \lambda \rbrace.
\end{displaymath}
These are $\mc{O}_v$-submodules of $V$ (by the Schwarz Inequality, see e.g. \cite[Lemma 9]{SpringerTameQuadratic}) and $V_\lambda/V^\circ_\lambda$ is naturally a $Kv$-vector space.
Now given $a \in K^\times$ with $v(a) = \lambda$ we can define an anisotropic quadratic form
$$\RF{a}{q} : V_\lambda/V^\circ_\lambda \to Kv : \ovl{x} \mapsto \ovl{a^{-1}q(x)}.$$
We call such a form $\RF{a}{q}$ a \emph{residue form} of $q$.

We will see now that residue forms completely determine the quadratic form in the non-dyadic case.
In the dyadic case, extra assumptions are required.
\begin{prop}\label{P:residueCharNot2}
Let $(K, v)$ be a non-dyadic henselian valued field, and fix a set $\Pi$ of representatives of $K^{\times}/v^{-1}(2vK)$. There is a group isomorphism
$ WK \to \bigoplus_{\pi \in \Pi} W(Kv) $
which maps the class of an anisotropic quadratic form $q$ to the class of $(\RF{\pi}{q})_{\pi \in \Pi}$.
In particular, two anisotropic quadratic forms over $K$ are isometric if and only if their residue forms are isometric.
\end{prop}
\begin{proof}
This is part of \autocite[Satz 3.1]{Tietze}.
\end{proof}
In the dyadic case, it might happen that, even when $q$ is itself non-singular, some of its residue forms are singular (i.e.~not non-singular).
\begin{prop}\label{P:TignolCharacteriseTameQuadratic}
Let $(K, v)$ be a dyadic henselian valued field, and fix a set $\Pi$ of representatives of $K^{\times}/v^{-1}(2vK)$.
Let $q$ be an anisotropic quadratic form over $K$. The following are equivalent:
\begin{enumerate}[$(i)$]
\item\label{it:inerthyp1} $q$ is hyperbolic over $K^{un}$,
\item\label{it:inerthyp2} $q \cong a_1 \llangle b_1 ]]_K \perp \ldots \perp a_n \llangle b_n ]]_K$ for some $n \in \nat$, $a_1, \ldots, a_n \in K^\times$ and $b_1, \ldots, b_n \in \mathcal{O}_v$,
\item\label{it:inerthyp3} $\RF{\pi}{q}$ is non-singular for all $\pi \in \Pi$.
\end{enumerate}
\end{prop}
\begin{proof}
We first observe that a quadratic extension $(L, w)/(K, v)$ is inert if and only if $L$ is the splitting field of a polynomial $T^2 - T - b$ with $b \in \mc{O}_v$; this follows from the fact that, since $\charac(Kv) = 2$, all polynomials of the form $T^2 - T - \alpha$ for $\alpha \in Kv$ are separable, and conversely every quadratic extension of $Kv$ is the splitting field of some polynomial of the form $T^2 - T - \alpha$ for $\alpha \in Kv$.
In particular, anisotropic $1$-fold Pfister forms of the form $\llangle b ]]_K$ for $b \in \mc{O}_v$ are precisely the norm forms of inert quadratic extensions of $K$.
With this in mind, the equivalence between \eqref{it:inerthyp1} and \eqref{it:inerthyp2} is given by \autocite[Corollary 17]{SpringerTameQuadratic}.

The equivalence between \eqref{it:inerthyp2} and \eqref{it:inerthyp3} is \autocite[Satz 4.1a]{Tietze}.
\end{proof}
For a dyadic valued field $(K, v)$, denote by $\IS{v}K$ the subgroup of $I_qK$ of Witt equivalence classes of quadratic forms satisfying property \eqref{it:inerthyp2} in \Cref{P:TignolCharacteriseTameQuadratic}.
By a direct approximation argument, one can verify that $\IS{v}K$ is the preimage of $\IS{v}K_v$ under the restriction homomorphism $I_qK \to I_qK_v$.
For a non-dyadic valued field, we simply set $\IS{v}K = I_qK$.
\begin{prop}\label{P:residueChar2}
Let $(K, v)$ be a dyadic henselian valued field, and fix a set $\Pi$ of representatives of $K^{\times}/v^{-1}(2vK)$.
There is a group isomorphism
$ \partial_v : \IS{v}K \to \bigoplus_{\pi \in \Pi} I_q Kv $
which maps the class of an anisotropic quadratic form $q$ in $\IS{v}K$ to the class of $(\RF{\pi}{q})_{\pi \in \Pi}$.
In particular, two anisotropic quadratic forms over $K$ whose residue forms are all non-singular, are isometric if and only if their residue forms are isometric.
\end{prop}
\begin{proof}
This is \autocite[Satz 4.1]{Tietze}.
\end{proof}

\begin{prop}\label{P:Pfister-when-good-slot}
Let $(K, v)$ be a dyadic henselian valued field, $d \in \nat^+$, $q$ an anisotropic $d$-fold Pfister form over $K$.
The following are equivalent.
\begin{enumerate}[(i)]
\item\label{it:good-slot1} $[q] \in \IS{v}K$,
\item\label{it:good-slot2} there exists an inert extension $L/K$ with $[L : K] \leq 2$ such that $q_L$ is isotropic,
\item\label{it:good-slot3} $q \cong \llangle a_1, \ldots, a_{d-1}, b]]_K$ for some $a_1, \ldots, a_{d-1}, b \in K^\times$ such that $v(b) = v(1+4b) = 0$.
\end{enumerate}
\end{prop}
\begin{proof}
$\eqref{it:good-slot2}\Leftrightarrow\eqref{it:good-slot1}$:
By \cite[Lemma 13 and Theorem 16]{SpringerTameQuadratic}, this
 is immediate from part \eqref{it:inerthyp1} of \Cref{P:TignolCharacteriseTameQuadratic} and the fact that isotropic Pfister forms are hyperbolic.

$\eqref{it:good-slot3} \Rightarrow\eqref{it:good-slot2}$: $\llangle a_1, \ldots, a_{d-1}, b]]_K$ is isotropic over the splitting field of $T^2 - T - b$, which is an inert extension of $K$ if $v(b) = v(1+4b) = 0$.

$\eqref{it:good-slot2}\Rightarrow\eqref{it:good-slot3}$: Given such an inert quadratic extension $L/K$, by lifting a representation of the residue field extension, we can find $b \in \mc{O}_v^\times$ with $1+4b \in \mc{O}_v^\times$ such that $L$ is the splitting field of $T^2 - T - b$.
By \cite[Proposition 22.11]{ElmanKarpenkoMerkurjev} $\llangle b ]]_K$ is similar to a subform of $q$, which by \cite[Proposition 24.1(1)]{ElmanKarpenkoMerkurjev} implies the existence of $a_1, \ldots, a_{d-1} \in K^\times$ such that $q \cong \llangle a_1, \ldots, a_{d-1}, b ]]_K$.
\end{proof}
We will now consider presentations of Pfister forms over valued fields, with the goal of understanding their residue forms, see \Cref{P:Pfister-residue-computation} below.
For the computations it will be useful to consider the $2$-torsion abelian group $vK/2vK$ as an $\ff_2$-vector space.
\begin{lem}\label{L:left-slot-finding}
Let $(K, v)$ be a valued field, $a_1, \ldots, a_{n+1} \in K^\times$.
If $v(a_{n+1}) \in \Span_{\ff_2} \lbrace v(a_1), \ldots, v(a_n) \rbrace \subseteq vK/2vK$, then there exists $a_1', \ldots, a_{n+1}' \in K^\times$ with $a_{n+1}' \in \mc{O}_v^\times$ and
$ \llangle a_1, \ldots, a_n, a_{n+1} \rrangle_K^b \cong \llangle a_1', \ldots, a_n', a_{n+1}' \rrangle_K^b$ as bilinear forms.
\end{lem}
\begin{proof}
We proceed by induction on $n$.
For $n = 0$ there is little to show: if $v(a_1) \in 2vK$, then there exists $b \in K^\times$ with $v(a_1) = 2v(b) = v(b^2)$, and then $\llangle a_1 \rrangle_K^b \cong \llangle a_1 b^{-2} \rrangle_K^b$, so we may set $a_1' = a_1b^{-2}$.

For the induction step, assume that $v(a_{n+1}) \in \Span_{\ff_2} \lbrace v(a_1), \ldots, v(a_n) \rbrace$ and that $v(a_{n+1}) \not\in 2vK$, otherwise we can conclude as before.
Since we may permute the $a_i$'s (using that $\llangle a_1, \ldots, a_i, a_{i+1}, \ldots, a_{n+1} \rrangle_K^b \cong \llangle a_1, \ldots, a_{i+1}, a_i, \ldots, a_{n+1} \rrangle_K^b$), we may assume without loss of generality that $v(-a_{n}a_{n+1}) = v(a_{n}) + v(a_{n+1}) \in \Span_{\ff_2} \lbrace v(a_1), \ldots, v(a_{n-1}) \rbrace$.
Now setting $a_1' = a_n + a_{n+1}$, $a_i' = a_{i-1}$ for $1 < i \leq n$ and $a_{n+1}' = -a_na_{n+1}$, we have $$\llangle a_1, \ldots, a_{n+1} \rrangle_K^b = \llangle a_1, \ldots, a_{n-1}, a_1', a_{n+1}' \rrangle_K^b = \llangle a_1', a_2', \ldots, a_{n}', a_{n+1}' \rrangle_K^b,$$
where we used the computation rules from \cite[Lemma 4.15]{ElmanKarpenkoMerkurjev}.
Note that $v(a_{n+1}') \in \Span_{\ff_2} \lbrace v(a_2'), \ldots, v(a_n') \rbrace$.
Hence we may apply the induction hypothesis to $\llangle a_{2}', \ldots, a_{n+1}' \rrangle_K^b$ to find $a_2'', \ldots, a_{n+1}'' \in K^\times$ with $v(a_{n+1}'') = 0$ such that $\llangle a_{2}', \ldots, a_{n+1}' \rrangle_K^b = \llangle a_{2}'', \ldots, a_{n+1}'' \rrangle_K^b$ and thus in particular $\llangle a_1, \ldots, a_{n+1} \rrangle_K^b = \llangle a_1', a_2'', \ldots, a_{n+1}'' \rrangle_K^b$, as desired.
\end{proof}
The following two propositions will prove useful for explicitly computing residue forms.
Here, we consider quadratic forms over a valued field which can be presented by homogeneous degree $2$ polynomials with coefficients in the valuation ring; we will use the terminology (an)isotropic, non-singular, etc.~for the quadratic form corresponding to the given polynomial, see \Cref{sect:QF-preliminaries}.
\begin{prop}\label{P:binaryFormValuations}
Let $(K, v)$ be a valued field, $n \in \nat$. Let $f(X_1, \ldots, X_n) \in \mathcal{O}_v[X_1, \ldots, X_n]$ be a quadratic form such that $\ovl{f} \in Kv[X_1, \ldots, X_n]$ is anisotropic.
For any elements $a_1, \ldots, a_n \in K$ we have that
\begin{displaymath}
v(f(a_1, \ldots, a_n)) = 2\min \lbrace v(a_i) \mid i \in \lbrace 1, \ldots, n \rbrace \rbrace.
\end{displaymath}
\end{prop}
\begin{proof}
See \cite[Proposition 4.2]{Andromeda-1}.
\end{proof}
\begin{prop}\label{P:anisotropicResidue}
Let $(K, v)$ be a henselian valued field, $q(X_1, \ldots, X_n) \in \mc{O}_v[X_1, \ldots, X_n]$ a quadratic form such that $\overline{q} \in Kv[X_1, \ldots, X_n]$ is non-singular.
Then $q$ is anisotropic if and only if $\overline{q}$ is anisotropic.
\end{prop}
\begin{proof}
The fact that $q$ is anisotropic when $\overline{q}$ is anisotropic follows from \Cref{P:binaryFormValuations}, without assuming that $q$ is non-singular or that $v$ is henselian.

For the other implication, assume that $\overline{q}$ is non-singular and isotropic over $Kv$. Taking $x, y \in \mc{O}_v^n$ with $\overline{x} \neq 0$, $\overline{q(x)} = 0$ and $\mf{b}_{\overline{q}}(\overline{x}, \overline{y}) \neq 0$, we obtain that
$$f(T) = q(x) + T\mf{b}_q(x, y) + T^2q(y) \in \mc{O}_v[T]$$
has a simple root $0$ in $Kv$.
Since $v$ is henselian, $f(T)$ thus has distinct roots $t_1, t_2$ in $K$.
But then either $x + t_1y \neq 0$ or $x + t_2y \neq 0$, and $q(x+t_1y) = q(x+t_2y) = f(t_1) = 0$, whereby $q$ is isotropic.
\end{proof}
\begin{prop}\label{P:Pfister-residue-computation}
Let $(K, v)$ be a henselian valued field, $a_1, \ldots, a_n \in K^\times$, $b \in K$ with $1+4b \neq 0$.
Assume that, for $m \in \nat$, $v(a_1), \ldots, v(a_m)$ are $\ff_2$-linearly independent in $vK/2vK$, and $a_{m+1}, \ldots, a_{n}, b, 1+4b \in \mc{O}_v^\times$.
Then $q = \llangle a_1, \ldots, a_n, b]]_K$ is anisotropic if and only if $\llangle \ovl{a_{m+1}}, \ldots, \ovl{a_n}, \ovl{b}]]_{Kv}$ is anisotropic.
Furthermore, in this case, we have for $\pi \in K^\times$ that $\RF{\pi}{q}$
\begin{itemize}
\item is $0$ if $v(\pi) \not \in \Span_{\ff_2} \lbrace v(a_1), \ldots, v(a_n) \rbrace$,
\item is similar to $\RF{1}{q} = \llangle \ovl{a_{m+1}}, \ldots, \ovl{a_n}, \ovl{b}]]_{Kv}$ if $v(\pi)\in \Span_{\ff_2} \lbrace v(a_1), \ldots, v(a_n) \rbrace$.
\end{itemize}
\end{prop}
\begin{proof}
We may write
$$ q \cong \bigperp_{I \subseteq \lbrace 1, \ldots, m \rbrace} \left((-1)^{\lvert I \rvert}\prod_{i \in I}a_i\right)\llangle a_{m+1}, \ldots, a_n, b ]]_K. $$
Since by assumption $a_{m+1}, \ldots, a_n, b, 1+4b \in \mc{O}_v^\times$, their residues in $Kv$ are all non-zero, hence $\llangle \ovl{a_{m+1}}, \ldots, \ovl{a_n}, \ovl{b}]]_{Kv}$ is a Pfister form, in particular non-singular.
By \Cref{P:anisotropicResidue}, if $q$ is anisotropic, then so is $\llangle \ovl{a_{m+1}}, \ldots, \ovl{a_n}, \ovl{b}]]_{Kv}$.
Furthermore, by the assumption on $v(a_1), \ldots, v(a_m)$, note that for $I, I' \subseteq \lbrace 1, \ldots, m \rbrace$ with $I \neq I'$ we have $v(\prod_{i \in I} a_i) \not\equiv v(\prod_{i \in I'} a_i) \bmod 2vK$.
Combining this with \Cref{P:binaryFormValuations} applied to the form $\llangle a_{m+1}, \ldots, a_n, b]]_K$, we conclude that $q$ is anisotropic if $\llangle \ovl{a_{m+1}}, \ldots, \ovl{a_n}, \ovl{b}]]_{Kv}$ is anisotropic, and in this case we see that for $\pi \in K^\times$ we have $\RF{\pi}{q} = 0$ if $v(\pi) \not\equiv v(\prod_{i \in I} a_i) \bmod 2vK$ for all $I \subseteq \lbrace 1, \ldots, m \rbrace$, and that $\RF{\pi}{q} \cong \overline{(\pi^{-1}(-1)^{\lvert I \rvert}(\prod_{i \in I} a_i)t^2)}\llangle \ovl{a_{m+1}}, \ldots, \ovl{a_n}, \ovl{b}]]_K$ if $v(\pi) = v(\prod_{i \in I} a_i) + 2v(t)$ for some $I \subseteq \lbrace 1, \ldots, m \rbrace$ and $t \in K^\times$.
From this the claimed statement follows.
\end{proof}

We can extend the notion of residue forms to be defined for arbitrary non-singular quadratic forms over valued fields $(K, v)$ as follows.
Firstly, for a non-singular (possibly isotropic) quadratic form $q$ over a henselian valued field $(K, v)$, we can find quadratic forms $q_{\text{an}}$ and $q_{\text{hyp}}$ over $K$ which are respectively anisotropic and hyperbolic, and such that $q \cong q_{\text{an}} \perp q_{\text{hyp}}$, and furthermore, the form $q_{\text{an}}$ is determined uniquely by $q$ up to isometry (special case of Witt's Decomposition Theorem, see \cite[Theorem 8.5]{ElmanKarpenkoMerkurjev}).
For $a \in K^\times$, we define the residue form $\RF{a}{q} = \RF{a}{q_{\text{an}}}$; this is uniquely defined up to isometry.
Finally, given a non-singular quadratic form $q$ over a valued field $(K, v)$ which is not necessarily henselian and $a \in K^\times$, we define $\RF{a}{q}$ to be $\RF{a}{q_{K_v}}$, i.e. the residue form of $q_{K_v}$ when considered as a form over the henselisation $K_v$.

For a valued field $(K, v)$ and $n \in \nat^+$, let us denote by $(\IS{v}K)^n$ the subgroup of $I_q^n K$ generated by scaled $n$-fold Pfister forms in $\IS{v}K$.
Note that $(\IS{v}K)^{n+1} \subseteq (\IS{v}K)^n \subseteq \IS{v}K \cap I^n_q K$.

\begin{stel}\label{T:Pfister-residue-main}
Let $(K, v)$ be a valued field, $\pi \in K^\times$, $n \in \nat^+$, $m \in \nat$, and assume that $[vK:2vK] = 2^m$.
We have that $\RF{\pi}{(\IS{v}K)^{n+m}} = I^n_q Kv$.
If $q$ is an $(n+m)$-fold Pfister form over $K$ with $[q] \in \IS{v}K$, then $\RF{1}{q}$ is an $n'$-fold Pfister form for some $n \leq n' \leq n+m$.

Furthermore, if we assume that $v$ is henselian and $I^{n+1}_q Kv = 0$, then $\RF{\pi}{q} \cong \RF{1}{q}$ for all anisotropic quadratic forms $q$ with $[q] \in (\IS{v}K)^{n+m}$.
In particular, in this case, the map
$$ (\IS{v}K)^{n+m} \to I^n_q Kv : [q] \mapsto [\RF{1}{q}], $$
is an isomorphism, restricting to a bijection between the classes of $(n+m)$-fold Pfister forms in $\IS{v}K$, and the classes of $n$-fold Pfister forms over $Kv$.
\end{stel}
\begin{proof}
Let $q$ be an $(n+m)$-fold Pfister form with $[q] \in \IS{v}(K)$.
If $q_{K_v}$ is isotropic, then $\RF{\pi}{q} = 0$ for all $\pi \in K^\times$, so assume now that $q_{K_v}$ is anisotropic.
By \Cref{P:Pfister-when-good-slot} we may write $q_{K_v} \cong \llangle a_1, \ldots, a_{n+m-1}, b]]_K$ for some $a_1, \ldots, a_{n+m-1}, b \in K_v^\times$ with $b, 1+4b \in \mc{O}_v^\times$.
In view of \Cref{L:left-slot-finding} we may replace the $a_i$'s and assume additionally that, for some $m' \leq m$, $v(a_{m'+1}) = \ldots = v(a_{n+m-1}) = 0$ and $v(a_1), \ldots, v(a_{m'})$ are $\ff_2$-linearly independent in $vK_v/2vK_v = vK/2vK$.
Setting $n' = n + m - m'$, we obtain by \Cref{P:Pfister-residue-computation} that, for every $\pi \in K^\times$, $\RF{\pi}{q} = \RF{\pi}{q_{K_v}}$ is either $0$ or similar to the $n'$-fold Pfister form $\RF{1}{q_{K_v}}$.
In particular, $[\RF{\pi}{q}] \in I_q^n Kv$.
Furthermore, since $\pi \in K^\times$ was arbitrary, we also have for $a \in K^\times$ that $[\RF{\pi}{aq}] = [\RF{\pi a^{-1}}{q}] \in I_q^n Kv$.
As $(\IS{v}K)^{n+m}$ is generated by scaled $(n+m)$-fold Pfister forms, we conclude that $\RF{\pi}{(\IS{v}K)^{n+m}} \subseteq I_q^n Kv$.

For the other inclusion, since the residue maps from Propositions \ref{P:residueCharNot2} and \ref{P:residueChar2} are group homomorphisms which are compatible with scaling by units, it suffices to show that the class of every $n$-fold Pfister form over $Kv$ can be obtained as the residue of an anisotropic $(n+m)$-fold Pfister form in $\IS{v}K$.
To this end, consider an $n$-fold Pfister form $q'$ over $Kv$, and write it as $q' = \llangle \ovl{a_{m+1}}, \ldots, \ovl{a_{n+m-1}}, \ovl{b} ]]_{Kv}$ for some $a_{m+1}, \ldots, a_{n+m-1}, b, 1+4b \in \mc{O}_v^\times$.
Let $a_1, \ldots, a_m$ be representatives of the different classes of $K^\times/v^{-1}(2vK)$.
By setting $q = \llangle a_1, \ldots, a_m, a_{m+1}, \ldots, a_{n+m-1}, b]]_K$, we obtain by \Cref{P:Pfister-residue-computation} that $\RF{1}{q} \cong q'$ as desired.

Now assume that $v$ is henselian and $I_q^{n+1} Kv = 0$.
Then by the above, for any anisotropic $(n+m)$-fold Pfister form $q$ with $[q] \in (\IS{v}K)^{n+m}$, all residue forms of $q$ must be pairwise similar $n$-fold Pfister forms, and in fact similar $n$-fold Pfister forms over $Kv$ must be isometric, so all residue forms of $q$ are isometric.
Since $(\IS{v}K)^{n+m}$ is generated by classes of $(n+m)$-fold Pfister forms, the rest of the statement follows readily by \Cref{P:residueCharNot2} and \Cref{P:residueChar2}.
\end{proof}
\begin{gev}\label{C:Pfister-residue-main}
Let $(K, v)$ be a valued field, $n \in \nat^+$, $m \in \nat$, and assume that $[vK : 2vK] = 2^m$.
There is a surjective group homomorphism
$$ \partial_v : (\IS{v}K)^{n+m}/(\IS{v}K)^{n+m+1} \to I_q^n Kv / I_q^{n+1} Kv : [q] \mapsto [\RF{1}{q}], $$
and it maps the class of an $(n+m)$-fold Pfister form $q$ over $K$ to the class of an $n$-fold Pfister form over $Kv$.
Furthermore, if $v$ is henselian and $I_q^{n+1} Kv = 0$, then $\partial_v$ is an isomorphism.
\end{gev}
\begin{proof}
By \Cref{T:Pfister-residue-main} we have a surjective group morphism $(\IS{v}K)^{n+m} \to I_q^n Kv$ mapping $(\IS{v}K)^{n+m+1}$ onto $I_q^{n+1} Kv$, hence by taking quotients one obtains the desired surjective group homomorphism $\partial_v$.
It also follows from \Cref{T:Pfister-residue-main} that if $v$ is henselian and $I_q^{n+1}Kv = 0$, then $(\IS{v}K)^{n+m+1} = 0$ and $\partial_v$ is an isomorphism.
\end{proof}

\section{Linkage over henselian valued fields}\label{sect:linkage-local}

\begin{prop}\label{P:linked-henselian}
Let $(K, v)$ be a valued field, $d \in \nat^+$. Assume that $[vK : 2vK] = 2^m$ for some $m \in \nat$. We have that
\begin{displaymath}
K \enspace \text{top-}(d+m)\text{-linked} \qquad \Rightarrow \qquad Kv \enspace \text{top-}d\text{-linked.}
\end{displaymath}
If $v$ is non-dyadic henselian, then the converse implication also holds.
\end{prop}
\begin{proof}
Assume first that $K$ is top-$(d+m)$-linked.
By \Cref{T:Pfister-residue-main} we have surjective group homomorphisms
\begin{eqnarray*}
(\IS{v}K)^{d+m} \to I_q^dKv \quad\text{and}\quad
(\IS{v}K)^{d+m+1} \to I_q^{d+1}K,
\end{eqnarray*}
mapping the class of a non-singular quadratic form $q$ to that of its first residue form $\RF{1}{q}$, and which in particular map a $(d+m)$-fold Pfister form to a $d$-fold Pfister form.
By the assumption that $K$ is $(d+m)$-linked, we have in particular $(\IS{v}K)^{d+m} = \IS{v}K \cap I_q^{d+m}K$.
Furthermore, since by assumption $0 = I_q^{d+m+1}K \supseteq (\IS{v}K)^{d+m+1}$, we also have $I_q^{d+1}K = 0$.
From this, the top-$d$-linkage of $Kv$ follows.

Now assume that $v$ is non-dyadic henselian and $Kv$ is top-$d$-linked.
Then $\IS{v}K = I_qK$ and the first homomorphism above reduces to a homomorphism $I_q^{d+m}K \to I_q^dK$ which, by \Cref{T:Pfister-residue-main}, is an isomorphism.
This shows that $K$ is top-$(d+m)$-linked if and only if $Kv$ is top-$d$-linked.
\end{proof}
\begin{gev}\label{C:strongly-linked-implies-linked}
Let $d \in \nat^+$, let $K$ be a field.
If $K$ is an $\Li{d}$-field, then any algebraic field extension of $K$ is top-$d$-linked.
\end{gev}
\begin{proof}
It suffices to show that any finite field extension of $K$ is top-$d$-linked.
An arbitrary finite field extension $L$ of $K$ is the residue field of a discrete valuation on $L(T)$, whereby the claim follows from \Cref{P:linked-henselian}.
\end{proof}

The following example was pointed out to me by Karim Johannes Becher.
\begin{vb}\label{vb-Li_n-vs-cd2}
As explained in \cite[2]{Becher-Biquaternion}, building on results of \cite{CTM-DelPezzo}, there exists a field $K$ of characteristic $0$ of $2$-cohomological dimension $1$ (i.e.~such that $I^{3}_q F =0$ for every function field in one variable $F/K$, see \Cref{P:cdFacts}) and such that there exists a biquaternion division algebra over $F = K(T)$, i.e.~a tensor product of two quaternion algebras which is not Brauer equivalent to a quaternion algebra.
Denoting by $\Br(F)[2]$ the $2$-torsion part of the Brauer group of $F$, there is a natural isomorphism $\Br(F)[2] \to I^2_qF$ mapping the class of a quaternion algebra to the class of its norm form \cite[Theorem 16.3]{ElmanKarpenkoMerkurjev}, and we conclude that $I^2_qF$ is not linked, so $K$ is not an $\Li 1$-field.
Setting $K_n = K(T_1, \ldots, T_{n-1})$, we thus obtain an example of a field $K_n$ satisfying $I^{n+2}_q F = 0$ for all function fields in one variable $F/K$ (i.e.~of $2$-cohomological dimension $n$), but $K_n$ is not an $\Li n$-field: we see by induction on $n$ using \Cref{P:linked-henselian} that $K_{n+1} = K_n(T_n)$ is not top-$(n+1)$-linked.
\end{vb}

\begin{lem}\label{L:Z-valued-maxdim-MMW}
Let $(K, v)$ be a dyadic henselian discretely valued field, $d \in \nat^+$.
Assume that $[Kv : Kv\pow{2}] \leq 2^{d-1}$.
Then $I_q^{d+1} K \subseteq \IS{v} K$.
\end{lem}
\begin{proof}
It suffices to show that $[q] \in \IS{v}K$ for any $(d+1)$-fold Pfister form $q$. By \Cref{P:Pfister-when-good-slot}\eqref{it:good-slot2} it suffices to show that $q \cong \llangle a_1, \ldots, a_d, c]]_K$ for some $a_1, \ldots, a_d, c \in K^\times$ with $c, 1+4c \in \mc{O}_v^\times$.
This is given by \cite[Proposition 3.9 and Proposition 4.12]{Andromeda-1}.
\end{proof}
\begin{prop}\label{P:linked-henselian-char2}
Let $(K, v)$ be a dyadic henselian valued field, $d \in \nat^+$.
Assume that $[vK : 2vK] \leq 2^m$ for some $m \in \nat$ and $[Kv : Kv\pow{2}] \leq 2^{d-1}$.
Assume that $I_q^{d+m} K \subseteq \IS{v} K$.
Then $K$ is top-$(d+m)$-linked.
\end{prop}
\begin{proof}
The assumption $[Kv : Kv\pow{2}] \leq 2^{d-1}$ implies that $K$ is top-$d$-linked by \Cref{P:AllSymbolsChar2} (and \Cref{C:strongly-linked-implies-linked}), in particular $I_q^{d+1}Kv = 0$.
It follows from \Cref{T:Pfister-residue-main} and the assumption $I_q^{d+m} K \subseteq \IS{v} K$ that the map $I_q^{d+m} K \to I_q^d Kv : [q] \mapsto [\RF{1}{q}]$ is an isomorphism.
It follows that $K$ is top-$(d+m)$-linked if and only if $Kv$ is top-$d$-linked.
\end{proof}

\section{Linkage over semi-global fields}
Given a valued field $(K, v)$, the \emph{rank} of the valuation (or of the valued field) is the Krull dimension of the valuation ring.
It is either a natural number, or $\infty$; in the second case a more refined notion of rank is given in \cite[Section 2.1]{Eng05}, but for our purposes the coarser notion will be sufficient.
Note that the trivial valuation is the only valuation of rank $0$, and a non-trivial valuation has rank $1$ if and only if its value group can be embedded as an ordered abelian group into $\rr$ \cite[Proposition 2.1.1]{Eng05}.

We now state the local-global principle for isotropy of quadratic forms which we will use.

\begin{stel}[V.~Mehmeti]\label{T:Mehmeti}
Let $(K,v)$ be a henselian rank $1$ valued field with $\charac(K) \neq 2$ and let $F/K$ be a function field in one variable.
Let $q$ be an anisotropic form of dimension at least $3$ over $F$.
Then $q$ is anisotropic over $F_w$ for some rank 1 valuation $w$ on $F$ such that either $w\vert_K$ is trivial or $w\vert_K = v$.
\end{stel}
\begin{proof}
See \cite[Corollary 3.19]{Mehmeti_PatchingBerkQuad} for the case where $(K, v)$ is complete; the extension from complete to henselian uses standard techniques and is explained e.g.~in \cite[Theorem 4.4]{BDGMZ}.
\end{proof}
If $(K, v)$ is a valued field and $F/K$ a function field in one variable, then the valuations $w$ on $F$ for which $w\vert_K = v$ can take on several forms; in particular, even when $v$ is discrete, some of the valuations $w$ on $F$ with $w\vert_K = v$ are not discrete.
They do, however, all satisfy the Dimension Inequality \autocite[Theorem 3.4.3]{Eng05}:
\begin{equation}\label{eq:dim-ineq}
\trdeg(Fw/Kv) + \rrk(wF/vK) \leq 1.
\end{equation}
Here, $\trdeg(Fw/Kv)$ denotes the transcendence degree of $Fw$ over $Kv$ and $\rrk(wF/vK)$ denotes the rational rank of the quotient group $wF/vK$, i.e.~the cardinality of a maximal linearly independent subset of $wF/vK$ considered as a $\zz$-module.
Furthermore, when equality holds in \eqref{eq:dim-ineq}, then $Fw$ is a finitely generated field extension of $Kv$ and $wF/vK$ is finitely generated as a $\zz$-module.

We shall use the following group-theoretic observation to relate the rational rank of $wF/vK$ to the index $[wF : 2wF]$.
This appears to be well-known, and implicitly contained e.g.~in \cite[Section 3]{Broecker}, but we include a proof for lack of a convenient reference.
\begin{prop}\label{P:quotient-torsionfree}
Let $G$ be a torsion-free abelian group and $H$ a subgroup of $G$ with $r = \rrk(G/H) < \infty$.
Let $n \in \nat^+$.
Then $[G : nG] \leq n^r [H : nH]$.
Furthermore, if $G/H$ is finitely generated, then $[G : nG] = n^r [H : nH]$.
\end{prop}
\begin{proof}
Consider first the case $r = 0$.
The hypothesis then implies that $G/H$ is torsion.
Furthermore, if additionally $G/H$ is finitely generated, then $G/H$ is finite.
In this case the statement is entirely contained in \cite[Lemma 3.4]{Becher-Leep}.

Now consider the general case.
We choose $S \subseteq G$ with $\lvert S \rvert = r$ and such that $\lbrace a + H \mid a \in S \rbrace$ is a maximal $\zz$-linearly independent subset of $G/H$.
Let $G_0$ denote the subgroup of~$G$ generated by $H \cup S$.
Then $G_0\simeq \zz^r\times H$, whereby
$[G_0 : nG_0] = n^r\cdot [H : nH]$.

The choice of $S$ implies that the quotient group $G/G_0$ is torsion.
Hence by the special case considered in the first paragraph we obtain that $[G:nG]\leq [G_0:nG_0]$.
If $G/H$ is finitely generated, then $G/G_0$ is finite, and again the special case considered in the first paragraph yields $[G:nG]= [G_0:nG_0]$.
\end{proof}

\begin{stel}\label{T:classKinduction}
  Let $d, m \in \nat, d \geq 1$. Assume that $(K,v)$ is a non-dyadic henselian valued field with $[vK : 2vK] \leq 2^m$.
  If $Kv$ is an $\Li{d}$-field, then $K$ is an $\Li{d+m}$-field.
\end{stel}
\begin{proof}
Since, for any $k \in \nat^+$, an $\Li{k}$-field is in particular an $\Li{k+1}$-field, it suffices to consider the case where $[vK : 2vK] = 2^m$, which we restrict to without loss of generality.
If $v$ is trivial, then there is nothing to show, so assume that $v$ is non-trivial, i.e.~has rank at least $1$.

Assume first that $v$ is of rank $1$.
Let $F/K$ be a function field in one variable and let $q$ be a non-singular quadratic form of dimension strictly greater than $2^{d+m+1}$ and such that $[q] \in I_q^{d+m+1}F$.
We will show that $q$ is isotropic.
In view of \Cref{P:d-linked-characterisation}, this suffices to conclude that $K$ is an  $\Li{d+m}$-field.

By \Cref{T:Mehmeti} it suffices to show that $q_{F_w}$ is isotropic for every rank $1$ valuation $w$ on $F$ for which either $w\vert_K = 0$ or $w\vert_K = v$.
We claim that for every valuation $w$ on $F$ for which either $w\vert_K = 0$ or $w\vert_K = v$, $F_w$ is top-$(d+m+1)$-linked.
The isotropy of $q_{F_w}$ then follows by \Cref{P:d-linked-characterisation}.

Let $w$ be a rank $1$ valuation on $F$ with $w\vert_K = 0$.
Then $w$ is a discrete valuation, and its residue field $Fw$ is a finite extension of $K$, hence itself a henselian valued field with respect to a valuation $v'$ extending $v$.
We have that $[v'(Fw) : 2v'(Fw)] = [vK : 2vK] = 2^m$ and that $(Fw)v'$ is a finite extension of $Kv$, hence $(Fw)v'$ is top-$d$-linked by \Cref{C:strongly-linked-implies-linked}.
It follows by \Cref{P:linked-henselian} that $Fw$ is top-$(d+m)$-linked and hence $F_w$ is top-$(d+m+1)$-linked.

Now let $w$ be a rank $1$ valuation on $F$ with $w\vert_K = v$.
By the Dimension Inequality \eqref{eq:dim-ineq} precisely one of two cases occurs.

In the first case, $Fw/Kv$ is a function field in one variable and $[wF : vK] < \infty$.
By \Cref{P:quotient-torsionfree} we have $[wF : 2wF] = [vK : 2vK] = 2^m$.
Since furthermore $Fw$ is top-$(d+1)$-linked it follows by \Cref{P:linked-henselian} that $F_w$ is top-$(d+m+1)$-linked.

In the second case, $Fw/Kv$ is an algebraic field extension and either $wF/vK$ is torsion, or $wF \cong vK \times \zz$ as an abelian group.
In either case, we have $[wF : 2wF] \leq 2[vK : 2vK] = 2^{m+1}$ by \Cref{P:quotient-torsionfree}.
In this case, since $Fw$ is top-$d$-linked by \Cref{C:strongly-linked-implies-linked}, it follows that $F_w$ is top-$(d+m+1)$-linked by \Cref{P:linked-henselian}.

This concludes the proof when $v$ is of rank $1$.
The extension from rank $1$ to arbitrary rank follows along the same lines as in the proof of \cite[Theorem 5.8]{BDGMZ}; we outline the main ideas.

First, suppose that the rank of $v$ is some natural number $n \geq 2$; we proceed by induction on $n$.
Let $w_1$ be a coarsest non-trivial coarsening of $v$.
Then $w_1$ is a rank $1$ valuation, and there exists a rank $n-1$ valuation $w_2$ on the residue field $Kw_1$ such that $v$ is a composition of $w_1$ and $w_2$, i.e.~such that $\mc{O}_v = \lbrace x \in \mc{O}_{w_1} \mid \ovl{x}^{w_1} \in \mc{O}_{w_2} \rbrace$.
See \cite[Section 2.3]{Eng05} for background on coarsenings, refinements, and compositions of valuations.
One has that $w_1$ and $w_2$ are both henselian, $(Kw_1)w_2 = Kv$, and $vK \cong w_1K \times w_2(Kw_1)$ as abelian groups, whereby in particular $[vK : 2vK] = [w_1K : 2w_1K][w_2(Kw_1) : 2w_2(Kw_1)]$.
We can apply the induction hypothesis to both $w_1$ and $w_2$, and infer the desired statement for $v$.

Finally, assume that $v$ is arbitrary, and let $F/K$ be a function field in one variable.
We rely again on \Cref{P:d-linked-characterisation}.
Let $q$ be a non-singular quadratic form of dimension strictly greater than $2^{d+m+1}$ and such that $[q] \in I_q^{d+m+1}F$.
We need to show that $q$ is isotropic; let us suppose for the sake of a contradiction that it is not.
We may first replace $K$ by its perfect hull $K^{per}$ and $F$ by $FK^{per}$: since $\charac(K) \neq 2$ we have that $K^{per}/K$ and $FK^{per}/F$ are a limit of purely inseparable odd degree extensions, hence $q$ remains anisotropic over $FK^{per}$ (by Springer's Theorem, see e.g.~\cite[Corollary 18.5]{ElmanKarpenkoMerkurjev}), and $K^{per}$ with the natural extension of $v$ still satisfies the hypotheses of the theorem; in particular $[vK^{per} : 2vK^{per}] \leq [vK : 2vK] \leq 2^m$.
So assume from now on, without loss of generality, that $K$ is perfect. 
Let $F_1$ be a finitely generated subfield of $F$ such that $q$ is defined over $F_1$ and $F_1/(F_1 \cap K)$ is of transcendence degree $1$.
Note that $v\vert_{F_1 \cap K}$ has finite rank, again by the Dimension Inequality.
Using that $K$ is perfect, there exists a subfield $K_0$ of $K$ containing $F_1 \cap K$ such that, with $v_0 = v\vert_{K_0}$, $(K_0, v_0)$ is henselian of finite rank, $K_0v_0 = Kv$, and $vK/v_0K_0$ is torsion-free: see e.g.~\cite[Corollary 3.16]{Kuhlmann-Tame} and its proof.
In particular, it follows that $[v_0K_0 : 2v_0K_0] \leq [vK : 2vK]$.
We conclude by the case considered in the previous paragraph that $K_0$ is an $\Li{d+m}$-field.
Setting $F_0 = F_1K_0$, we have that $q$ is defined over $F_0$ and $[q] \in I_q^{d+m+1}F_0$, whereby $q$ is isotropic, as we had to show.
\end{proof}

\begin{opm}\label{R:local-global-technique}
The strategy used in the first part of the proof of \Cref{T:classKinduction} (when $v$ is of rank $1$) is not new.
Local-global principles based on patching prior to the one from \Cref{T:Mehmeti} have been used in the past to study problems about quadratic forms over function fields over henselian discretely valued fields by passing to completions.

For example, in \cite[Theorem 4.10]{HHK_ApplicationsPatchingToQuadrFormsAndCSAs}, the following was shown: if $(K, v)$ is a non-dyadic henselian discretely valued field and $e \in \nat^+$ is such that, for any function field in one variable $F/Kv$, one has that all quadratic forms in $e+1$ variables over $F$ are isotropic (or in the terminology of \cite{HHK_ApplicationsPatchingToQuadrFormsAndCSAs}: the \textit{strong $u$-invariant of $Kv$} is at most $e$), then it follows that, for every function field in one variable $F/K$, all quadratic forms in $2e+1$ variables over $F$ are isotropic.
By iterating this argument, one can then obtain the following: if $K_1, K_2, \ldots, K_d$ is a sequence of fields, where $K_1$ is finite of odd characteristic, and each $K_{i+1}$ is complete with respect to a discrete valuation with residue field $K_i$, then, for any function field in one variable $F/K_d$, all quadratic forms in $2^{d+1} + 1$ variables over $F$ are isotropic (see \cite[Corollary 4.14]{HHK_ApplicationsPatchingToQuadrFormsAndCSAs}).
Combining this result with \Cref{P:d-linked-characterisation}, we obtain that $K_d$ is an $\Li d$-field.
However, there seems to be no way known to make this argument work when $\charac(K_1) = 2$.
We present a way to nevertheless obtain the linkage result in \Cref{T:higher-local}.
\end{opm}

\begin{lem}\label{L:dyadic-sepclosed}
Let $(K, v)$ be a dyadic henselian discretely valued field with $\charac(K) = 0$, $d \in \nat^+$ such that $[Kv : (Kv)\pow{2}] \leq 2^{d-1}$ and $Kv$ is separably closed.
Then for every field extension $F/K$ of transcendence degree one, $I_q^{d+2} F = 0$.
\end{lem}
\begin{proof}
Since $Kv$ is separably closed, we have $I_q Kv = 0$, so in particular $I_q^d Kv = 0$.
By \Cref{L:Z-valued-maxdim-MMW} and \Cref{T:Pfister-residue-main} we have that $I^{d+1}_q K \cong I^{d}_q Kv = 0$.

Every finite extension $L/K$ is again a dyadic henselian valued field with $\charac(L) = 0$ and with $[Lw : (Lw)\pow{2}] \leq 2^{d-1}$ and $Lw$ separably closed, for $w$ the unique extension of $v$ to $L$.
So $I^{d+1}_q L = 0$ follows by the argument from the previous paragraph.
Now the desired statement follows from \Cref{P:cdFacts}.
\end{proof}
\begin{stel}\label{T:classKinduction-char2}
Let $d \in \nat^+$.
Let $(K, v)$ be a dyadic henselian discretely valued field, $\charac(K) = 0$.
If $[Kv : (Kv)\pow{2}] \leq 2^{d-1}$, then $K$ is an $\Li{d+1}$-field.
\end{stel}
\begin{proof}
Let $F/K$ be a function field in one variable and let $q$ be a non-singular quadratic form of dimension strictly greater than $2^{d+2}$ and such that $[q] \in I_q^{d+2}F$.
We will show that $q$ is isotropic.
In view of \Cref{P:d-linked-characterisation}, this suffices to conclude that $K$ is an  $\Li{d+1}$-field.

By \Cref{T:Mehmeti} it suffices to show that $q_{F_w}$ is isotropic for every rank $1$ valuation $w$ on $F$ for which either $w\vert_K = 0$ or $w\vert_K = v$.
We claim that for every valuation $w$ on $F$ for which either $w\vert_K = 0$ or $w\vert_K = v$, $F_w$ is top-$(d+1)$-linked.
The isotropy of $q_{F_w}$ then follows by \Cref{P:d-linked-characterisation}.

Let $w$ be a rank $1$ valuation on $F$ with $w\vert_K = 0$.
Then $w$ is a discrete valuation, and its residue field $Fw$ is a finite extension of $K$, hence itself a henselian valued field with respect to a discrete valuation $v'$ extending $v$.
We have that $(Fw)v'$ is a finite extension of $Kv$, hence $[(Fw)v' : ((Fw)v')\pow{2}] \leq 2^{d-1}$.
It follows by \Cref{L:Z-valued-maxdim-MMW} and \Cref{P:linked-henselian-char2} that $Fw$ is top-$(d+1)$-linked and hence by \Cref{P:linked-henselian} $F_w$ is top-$(d+2)$-linked.

Now let $w$ be a rank $1$ valuation on $F$ with $w\vert_K = v$.
We first argue that $I_q^{d+2} F \subseteq \IS{w} F$.
In view of \Cref{P:TignolCharacteriseTameQuadratic} we equivalently need to show that $I_q^{d+2} F^{un}_w = 0$.
The subfield $K^{un}$ of $F^{un}_w$ has a separably closed residue field $K^{un}v^{un}$ with $[K^{un}v^{un} : (K^{un}v^{un})\pow{2}] \leq 2^{d-1}$, and $F^{un}_w$ is an algebraic extension of $FK^{un}$ and thus of transcendence degree $1$ over $K^{un}$, hence $I_q^{d+2} F^{un}_w = 0$ by \Cref{L:dyadic-sepclosed}.
This concludes the proof of the claim.

By the Dimension Inequality \eqref{eq:dim-ineq} precisely one of two cases occurs.
In the first case, $Fw/Kv$ is a function field in one variable and $w$ is a discrete valuation.
In this case, since $[Fw : (Fw)\pow{2}] \leq 2[Kv:(Kv)\pow{2}] \leq 2^d$, it follows by \Cref{P:linked-henselian-char2} that $F_w$ is top-$(d+2)$-linked.

In the second case, $Fw/Kv$ is an algebraic field extension and either $wF/vK$ is torsion, or $wF \cong \zz^2$ as an abelian group.
In either case, we have by \Cref{P:quotient-torsionfree} that $[wF : 2wF] \leq 2[vK : 2vK] = 4$ and $[Fw : (Fw)\pow{2}] \leq [K : K\pow{2}] \leq 2^{d-1}$, so it follows by \Cref{P:linked-henselian-char2} that $F_w$ is top-$(d+2)$-linked.
\end{proof}
It would be worthwhile to investigate to what extent one can weaken the hypotheses of \Cref{T:classKinduction-char2}, e.g.~replacing ``$[Kv : (Kv)\pow{2}]\leq 2^{d-1}$ by ``$Kv$ is is an $\Li{d}$-field'', or admitting other types of valuations as in \Cref{T:classKinduction}.

For completeness, we remark that when the valued field itself has characteristic $2$, very little can be said in general:
\begin{opm}\label{R:complete-char2}
For a valued field $(K, v)$ with $\charac(K) = 2$, we have $[K : K\pow{2}] \geq [vK : 2vK][Kv : Kv\pow{2}]$.
However, for any field $\kappa$ of characteristic $2$, any ordered abelian group $\Gamma$, and any $d \in \nat^+$, one can find a henselian valued field $(K, v)$ of characteristic $2$ with $vK = \Gamma$, $Kv = \kappa$ and $[K : K\pow{2}] > 2^d$.
Thus in view of \Cref{P:AllSymbolsChar2}, given a henselian valued field $(K, v)$ of characteristic $2$, nothing positive can be inferred on whether $K$ is an $\Li{d}$-field for some given $d$, just by knowing $Kv$ and $vK$.
\end{opm}

Following \cite{Kuhlmann-Tame} a henselian valued field $(K, v)$ field is called \emph{inseparably defectless} if for every inseparable finite extension $(L, w)$ one has $[L : K] = [wL : wK][Lw : Kv]$.
In particular, if $(K, v)$ is inseparably defectless and of characteristic $2$, then $[K : K\pow{2}] = [vK : 2vK][Kv : Kv\pow{2}]$ holds.
A complete discretely valued field is always inseparably defectless, see e.g.~\cite[\nopp 16:4]{OMe00}.
We thus obtain the following.
\begin{prop}\label{P:complete-char2}
Let $d \in \nat^+$.
Let $(K, v)$ be a complete discretely valued field, $\charac(K) = 2$.
If $[Kv : (Kv)\pow{2}] \leq 2^{d-1}$, then $[K : K\pow{2}] \leq 2^d$, and in particular $K$ is an $\Li{d+1}$-field.
\end{prop}

We conclude with some applications of \Cref{T:classKinduction} and \Cref{T:classKinduction-char2}.
Concretely, we may apply these theorems to complete discretely valued fields with residue fields mentioned in \Cref{E:linkage}.
\begin{vbn}\label{E:Ld-fields-new}
\begin{enumerate}
\item $\qq(\sqrt{-1})(\!(T)\!)$ is an $\Li{3}$-field.
\item If $K$ is pseudo-algebraically closed, and furthermore perfect if $\charac(K) = 2$, then $K(\!(T)\!)$ is an $\Li{2}$-field.
\end{enumerate}
\end{vbn}
As promised, we obtain in particular the following result on higher local fields.
\begin{stel}\label{T:higher-local}
Let $d \in \nat^+$ and suppose that there is a sequence $K_1, K_2, \ldots, K_d$ of fields, where $K_1$ is finite, and each $K_{i+1}$ is complete with respect to a discrete valuation with residue field $K_i$.
Then $K_d$ is an $\Li{d}$-field.
\end{stel}
\begin{proof}
If $\charac(K_1) \neq 2$, then $\charac(K_i) \neq 2$ for all $i=1, \ldots, d$.
Furthermore, $K_1$ is an $\Li{1}$-field by \Cref{E:linkage}\eqref{it:Cd}.
The result then follows by applying \Cref{T:classKinduction} $d-1$ times.

Now assume that $\charac(K_1) = 2$.
There exists $j \in \lbrace 1, \ldots, d \rbrace$ such that $\charac(K_i) = 2$ for $i = 1, \ldots, j$ and $\charac(K_i) = 0$ for $i = j+1, \ldots, d$.
In view of \Cref{P:complete-char2} we have $[K_j : K_j\pow{2}] = 2^{j-1}$.
If $j = d$, then the desired statement follows from \Cref{P:AllSymbolsChar2} and we are done.
Finally, if $j < d$, then $K_{j+1} \in \Li{j+1}$ by \Cref{T:classKinduction-char2}, and then $K_d$ is an $\Li{d}$-field by \Cref{T:classKinduction} (applied $d-j-1$ times).
\end{proof}

Looking onward, if we want to understand the class of $\Li{d}$-fields, a crucial question is how the property behaves under transcendental extensions.
More specifically:
\begin{ques}
Let $d \in \nat$ and let $K$ be an $\Li{d}$-field.
Is $K(T)$ an $\Li{d+1}$-field?
\end{ques}

\printbibliography
\end{document}